\documentclass[12pt,a4paper]{article}
\usepackage{graphicx,amsmath,amssymb,latexsym,epsfig}
\usepackage{amsthm}
\newtheorem{theorem}{Theorem}

\newtheorem{lemma}{Lemma}
\newtheorem{proposition}{Proposition}

\def\PP{{\mathbb P}}

\def\d{{\mathrm{d}}}

\def\R{{\mathbb R}}

\def\N{{\mathbb N}}

\def\cF{{\mathfrak F}}

\def\d{d}

\title{The M\'ezard-Parisi equation for matchings in pseudo-dimension $d>1$}
\author{Justin Salez}
\begin{document}
\maketitle
\begin{abstract}
We establish existence and uniqueness of the solution to the cavity equation for the random assignment problem in pseudo-dimension $d>1$, as conjectured by Aldous and Bandyopadhyay (Annals of Applied Probability, 2005) and W\"astlund (Annals of Mathematics, 2012). This fills the last remaining gap in the proof of the original M\'ezard-Parisi prediction for this problem (Journal de Physique Lettres, 1985). 
\vspace{0.3cm}\vspace{0.3cm}
\end{abstract}
\section{Introduction}
The \emph{random assignment problem} is a now classical problem in probabilistic combinatorial optimization. Given an $n\times n$ array $\{X_{i,j}\}_{1\leq i,j\leq n}$ of \textsc{iid} non-negative random variables, it asks about the statistics of
\begin{eqnarray*}
M_n & := &\min_{\sigma}\sum_{i=1}^nX_{i,\sigma(i)},
\end{eqnarray*}
where the minimum runs over all permutations $\sigma$ of $\{1,\ldots,n\}$. This corresponds to finding a minimum-length perfect matching on the complete bipartite graph $K_{n,n}$ with edge-lengths $\{X_{i,j}\}_{1\leq i,j\leq n}$. 
Using the celebrated \emph{replica symmetry ansatz} from statistical physics, M\'ezard and Parisi \cite{MePa85,MePa86,MePa87} made a remarkably precise prediction concerning the regime where $n$ tends to infinity while the distribution of $X_{i,j}$ is kept fixed and satisfies 
\begin{eqnarray*}
\PP\left(X_{i,j}\leq x\right) \ \sim \ x^{\d} & \textrm{ as } & x\to 0^+,
\end{eqnarray*}
for some exponent $0<\d<\infty$. Specifically,  they conjectured that 
\begin{eqnarray}
\label{prediction}
\frac{M_n}{n^{1-1/\d}} & \xrightarrow[n\to\infty]{\PP} & -\d\int_{\R}f(x)\ln{f(x)}dx,
\end{eqnarray}
where the function $f\colon\R\to[0,1]$ solves the so-called \emph{cavity equation}:
\begin{eqnarray}
\label{rde}
f(x) & = & \exp\left(-\int_{-x}^{+\infty}\d(x+y)^{\d-1}f(y)dy\right).
\end{eqnarray}
Aldous \cite{Al92,Al01} proved this conjecture in the special case $\d=1$, where the term $(x+y)^{\d-1}$ simplifies and makes the cavity equation exactly solvable, yielding
\begin{eqnarray*}
f(x) =\frac{1}{1+e^x} & \textrm{ and } & -\d\int_{\R}f(x)\ln{f(x)}dx=\frac{\pi^2}{6}.
\end{eqnarray*}
Since then, several alternative proofs have been found \cite{LiWa04,NaBaSh05,Wa09}. 
This stands in sharp contrast with the case $\d\neq 1$, where showing that the M\'ezard-Parisi equation (\ref{rde}) admits a unique solution has until now remained an open problem \cite[Open Problem 63]{AlBa05}. W\"astlund \cite{Wa12} circumvented this issue by considering instead the truncated equation 
\begin{eqnarray}
\label{truncated}
f_\lambda(x) & = & \exp\left(-\int_{-x}^{\lambda}\d(x+y)^{\d-1}f_\lambda(y)dy\right),\qquad 0<\lambda<\infty.
\end{eqnarray}
Using an ingenious game-theoretical interpretation of this equation, he showed the existence of a unique, global attractive solution $f_\lambda\colon[-\lambda,\lambda]\to[0,1]$ for every $0<\lambda <\infty$, provided $d\geq 1$. He then used this fact to establish that
\begin{eqnarray}
\label{interchange}
\frac{M_n}{n^{1-1/\d}} & \xrightarrow[n\to\infty]{\PP} & {\lim_{\lambda\to+\infty}} \uparrow -\d\int_{-\lambda}^\lambda f_\lambda(x)\ln{f_\lambda(x)}dx.
\end{eqnarray}
W\"astlund \cite{Wa12} explicitly left open the problem of completing the proof of the original M\'ezard-Parisi prediction by showing (i) that  the untruncated cavity equation admits a unique solution $f$ and (ii) that $f_\lambda\to f$ as $\lambda\to\infty$. The purpose of this short paper is to establish this conjecture.

\begin{theorem}
\label{th:main}
For $\d>1$, the M\'ezard-Parisi equation (\ref{rde}) admits a unique solution $f\colon\R\to[0,1]$. Moreover, $f_\lambda\to f$ pointwise as $\lambda\to+\infty$, and
\begin{eqnarray*}
\int_{-\lambda}^\lambda f_\lambda(x)\ln{f_\lambda(x)}dx & \xrightarrow[\lambda\to+\infty]{} & \int_{\R}f(x)\ln{f(x)}dx.
\end{eqnarray*}
Consequently, the two limits in (\ref{prediction}) and (\ref{interchange}) coincide. 
\end{theorem}
In addition, we provide a short alternative proof of the crucial result of \cite{Wa12} that the truncated equation (\ref{truncated}) admits a unique, attractive solution. 
\paragraph{Remark 1.}
Very recently, a proof of uniqueness for the truncated equation (\ref{truncated}) has been announced  \cite{La14} for the case $0<\d<1$. It would be interesting to see if the result of the present paper can be extended to this regime. 
\paragraph{Remark 2.}
For a random variable $Z$ with $\PP\left(Z>x\right)=f(x)$, the cavity equation (\ref{rde}) simply expresses the fact that $Z$ solves the distributional identity 
\begin{eqnarray}
\label{distrib}
Z & \stackrel{d}{=} & \min_{i\geq 1}\left\{\xi_i-Z_i\right\},
\end{eqnarray}
where $\{\xi_i\}_{i\geq 1}$ is a Poisson point process with intensity $\d x^{\d-1}\partial x$ on $[0,\infty)$, and $\{Z_i\}_{i\geq 1}$ are \textsc{iid} with the same distribution as $Z$, independent of $\{\xi_i\}_{i\geq 1}$. Such \emph{recursive distributional equations} arise naturally in a variety of models from statistical physics, and the question of existence and uniqueness of solutions plays a crucial role for the rigorous understanding of those models. We refer the interested reader to the comprehensive surveys \cite{Al04,AlBa05} for more details. In particular, \cite[Section 7.4]{AlBa05} contains a detailed discussion on equation (\ref{distrib}), and \cite[Open Problem 63]{AlBa05} raises explicitly the uniqueness issue.  We note that the refined question of \emph{endogeny} remains a challenging open problem. Recursive distributional equations for other mean-field combinatorial optimization problems have been analysed in e.g. \cite{GaNoSw06,PaWa12,Kh14}. 

\paragraph{}The remainder of the paper is organized as follows. Section 2 deals with the truncated equation (\ref{truncated}) for fixed $0<\lambda<\infty$ and is devoted to the alternative analytical proof that there is a unique, globally attractive solution $f_\lambda$. Section 3 prepares the $\lambda\to\infty$ limit by providing uniform controls on  the family $\{f_\lambda\colon0<\lambda<\infty\}$ and by characterizing the possible limit points. This reduces the proof of Theorem \ref{th:main} to establishing uniqueness in the un-truncated M\'ezard-Parisi equation ($\lambda=\infty$), which is done in Section 4.

\section{The truncated cavity equation $(\lambda<\infty)$}
Fix a parameter $0<\lambda<\infty$. On the set $\cF$ of non-increasing functions $f\colon[-\lambda,{\lambda}]\to[0,1]$, define an operator $T$ by
\begin{eqnarray}
\label{T}
(T f)(x) & = & \exp\left(-\d\int_{-x}^\lambda(x+y)^{\d-1} f(y)dy\right).
\end{eqnarray}
The purpose of this section is to give a short and purely analytical proof of the following result, which was the main technical ingredient in \cite{Wa12} and was therein established using an ingenious game-theoretical framework. 

\begin{proposition}
\label{pr:truncated}
$T$ admits a unique fixed point $f_\lambda$ and it is attractive in the sense that $|T^n f(x)-f_\lambda(x)|\xrightarrow[n\to\infty]{} 0,$ uniformly in both $x\in[-\lambda,\lambda]$ and $f\in\cF$.
\end{proposition}
\begin{proof}
Write $f\leq g$ to mean $f(x)\leq g(x)$ for all $x\in[-\lambda,\lambda]$. 
In particular, 
\begin{eqnarray*}
{\bf 0}\leq f \leq T{\bf 0}
\end{eqnarray*}
 for every $f\in\cF$, 
where $\bf{0}$ denotes the constant-zero function. Note also that the operator $T$ is  non-increasing, in the sense that
 \begin{eqnarray*}
f\leq g & \Longrightarrow & T f\geq T g. 
\end{eqnarray*}
Those two observations imply that the sequences $\{T^{2n}{\bf 0}\}_{n\geq 0}$ and $\{T^{2n+1}{\bf 0}\}_{n\geq 0}$ are respectively non-decreasing and non-increasing, and that their respective pointwise limits $f^{-}$ and $f^+$ satisfy
 \begin{eqnarray*}
f^- \ \leq \ \liminf_{n\to\infty}T^n f & \leq & \limsup_{n\to\infty}T^n f \ \leq \ f^+,
\end{eqnarray*}
for any $f\in\cF$. 
Moreover, the dominated convergence Theorem ensures that $T$ is continuous with respect to pointwise convergence, allowing to pass to the limit in the identity $T^{n+1} {\bf 0}=T(T^n{\bf 0})$ to deduce that
 \begin{eqnarray}
 \label{pm}
Tf^- = f^+ & \textrm{ and } & Tf^+ = f^-.
\end{eqnarray}
Therefore, the proof boils down to the identity $f^-=f^+$, which we now establish. 
 By definition, we have for any $f\in\cF$,  
  \begin{eqnarray*}
(T f)(x) & = & \exp\left(-\d\int_{-\lambda}^{\lambda} (x+y)^{\d-1}{\bf 1}_{(x+y\geq 0)}f(y)dy\right).
\end{eqnarray*}
Since $\d>1$, we may differentiate under the integral sign to obtain
 \begin{eqnarray*}
(T f)'(x) & = & -\d(\d-1)(T f)(x)\int_{-\lambda}^{\lambda} (x+y)^{\d-2}{\bf 1}_{(x+y\geq 0)} f(y)dy.
\end{eqnarray*}
Integrating over $\left[-\lambda,\lambda\right]$ and noting that $(T f)\left(-\lambda\right)=1$, we conclude that
 \begin{eqnarray*}
1-(T f)\left(\lambda\right)& = & \d(\d-1)\iint_{\left[-\lambda,\lambda\right]^2} (x+y)^{\d-2}{\bf 1}_{(x+y\geq 0)}(T f)(x)f(y)dx dy.
\end{eqnarray*}
Let us now specialize to $f=f^\pm$. In both cases, the right-hand side is 
 \begin{eqnarray*}
 \d(\d-1)\iint_{\left[-\lambda,\lambda\right]^2} (x+y)^{\d-2}{\bf 1}_{(x+y\geq 0)}f^+(x)f^-(y)dx dy,
\end{eqnarray*}
by (\ref{pm}). Therefore, we have $(T f^+)\left(\lambda\right)=(T f^-)(\lambda)$, i.e. 
 \begin{eqnarray*}
 \int_{-\lambda}^{\lambda} \d(\lambda+y)^{\d-1}f^+(y)dy & = &  \int_{-\lambda}^{\lambda} \d(\lambda+y)^{\d-1}f^-(y)dy.
 \end{eqnarray*}
 Since we already know that $f^-\leq f^+$, this forces $f^-= f^+$ almost-everwhere on $[-{\lambda},\lambda]$, and hence everywhere by continuity. Finally, the convergence $T^n{\bf 0}\to f_\lambda=f^\pm$ is automatically uniform on $[-\lambda,\lambda]$, by Dini's Theorem.
\end{proof}

\section{Relative compactness of solutions $(\lambda\to\infty)$}
In order to study properties of the family $\{f_\lambda\colon 0<\lambda<\infty\}$, we extend the domain of $f_\lambda$ to $\R$ by setting $f_\lambda(x)=1$ for $x\leq -\lambda$ and $f_\lambda(x)=0$ for $x>\lambda$. 

\begin{proposition}[Uniform bounds]
\label{pr:uniform}
For all $0<\lambda<\infty$ and $x\geq 0$,
\begin{eqnarray*}
f_\lambda(x) & \leq &  \exp\left(-\frac{x^\d}{e}\right)\\
1-f_\lambda(-x) & \leq &  \exp\left(-\frac{x^\d}{e}\right)\\
f_\lambda(-x)\ln\frac{1}{f_\lambda(-x)}  & \leq & \exp\left(-\frac{x^\d}{e}\right)\\
f_\lambda(x)\ln\frac{1}{f_\lambda(x)}  & \leq & \left(1+ \frac{x^\d}{e}\right)\exp\left(-\frac{x^\d}{e}\right).
\end{eqnarray*}
\end{proposition}
\begin{proof}
Let $0<\lambda<\infty$. We may assume that $x\in[0,\lambda]$, otherwise the above bounds are trivial. By definition, we have
\begin{eqnarray}
\label{flambda}
f_\lambda(x) & = &  \exp\left(-\int_{-x}^\lambda\d(x+y)^{\d-1}f_\lambda(y)dy\right).
\end{eqnarray}
Now, since $x\geq 0$ and $f_\lambda$ is non-increasing, we have
\begin{eqnarray*}
\int_{-x}^{\lambda}(x+y)^{\d-1}f_\lambda(y)dy & = & \int_{-x}^0(x+y)^{\d-1}f_\lambda(y)dy+\int_{0}^\lambda (x+y)^{\d-1}f_\lambda(y)dy\\
& \geq & f_\lambda(0)\frac{x^\d}{\d}+\int_{0}^\lambda  y^{\d-1}f_\lambda(y)dy.
\end{eqnarray*}
Applying $u\mapsto \exp(-\d u)$ to both sides and using (\ref{flambda}), we obtain
\begin{eqnarray}
\label{upperb}
f_\lambda(x) & \leq & f_\lambda(0)\exp(-f_\lambda(0)x^\d).
\end{eqnarray}
In turn, this inequality implies that for all $x\geq 0$,
\begin{eqnarray*}
\int_{x}^{\lambda}\d(y-x)^{\d-1}f_\lambda(y)dy & \leq & f_\lambda(0)\int_{x}^{+\infty}\d y^{\d-1}e^{-f_\lambda(0)y^\d}dy
\ = \ 
\exp(-f_\lambda(0)x^\d).
\end{eqnarray*}
Applying $u\mapsto \exp(-u)$ to both sides, we conclude that
\begin{eqnarray}
\label{lowerb}
f_\lambda(-x) & \geq & \exp\left(-e^{-f_\lambda(0)x^\d}\right).
\end{eqnarray}
In particular, taking $x=0$ yields $f_\lambda(0)\geq e^{-1}$, and reinjecting 
this into (\ref{upperb}) and (\ref{lowerb}) easily yields the first three claims. For the last one, observe that $u\mapsto u\ln\frac{1}{u}$ increases on $[0,e^{-1}]$ and decreases on $[e^{-1},1]$,  with the value at $u=e^{-1}$ being precisely $e^{-1}$. Therefore, if $\exp(-x^\d/e)\leq e^{-1}$, we may use the bound $f_\lambda(x)\leq \exp(-x^\d/e)$ to deduce that
\begin{eqnarray*}
f_\lambda(x)\ln \frac{1}{f_\lambda(x)} & \leq & \frac{x^\d}{e}\exp\left(-\frac{x^\d}e\right).
\end{eqnarray*}
On the other hand, if $\exp(-x^\d/e)\geq e^{-1}$, then 
\begin{eqnarray*}
f_\lambda(x)\ln \frac{1}{f_\lambda(x)} & \leq & e^{-1}\ \leq \ \exp\left(-\frac{x^\d}{e}\right).
\end{eqnarray*}
In both cases, the last inequality holds, and the proof is complete.
\end{proof}

\begin{proposition}
\label{pr:tight}
The family $\left\{f_\lambda\colon 0<\lambda<\infty\right\}$ is relatively compact with respect to the topology of uniform convergence on $\R$, and any sub-sequential limit as $\lambda\to\infty$ must solve the cavity equation (\ref{rde}). 
\end{proposition}
\begin{proof}Let $\{\lambda_n\}_{n\geq 0}$ be any sequence of positive numbers such that $\lambda_n\to\infty$ as $n\to\infty$. By Helly's compactness principle for uniformly bounded monotone functions (see e.g. \cite[Theorem 36.5]{KoFo75}), there exists an increasing sequence $\{n_k\}_{k\geq 0}$ in $\N$ and a non-increasing function $f\colon\R\to[0,1]$ such that 
\begin{eqnarray}
\label{helly}
f_{\lambda_{n_k}}(x) & \xrightarrow[k\to\infty]{} & f(x),
\end{eqnarray}
for all $x\in\R$. Thanks to the first inequality in Proposition \ref{pr:uniform}, we may invoke dominated convergence to deduce that for each $x\in\R$, 
\begin{eqnarray*}
\int_{-x}^{\lambda_{n_k}}f_{\lambda_{n_k}}(y)(x+y)^{\d-1}dy & \xrightarrow[k\to\infty]{} & \int_{-x}^{+\infty}f(y)(x+y)^{\d-1}dy.
\end{eqnarray*}
Applying $u\mapsto\exp(-\d u)$ and recalling (\ref{flambda}), we see that
\begin{eqnarray*}
f(x) & = & \exp\left(-\d\int_{-x}^{+\infty}f(y)(x+y)^{\d-1}dy\right),
\end{eqnarray*}
which shows that $f$ must solve the cavity equation (\ref{rde}). This identity easily implies that $f$ is continuous. Consequently, the convergence (\ref{helly}) is uniform in $x\in\R$, by Dini's Theorem. 
\end{proof}

\section{The un-truncated cavity equation $(\lambda=\infty)$}

To conclude the proof of Theorem \ref{th:main}, it now remains to show that the un-truncated equation
\begin{eqnarray}
\label{Tinfty}
f(x) & = & \exp\left(-\d\int_{-x}^{+\infty}(x+y)^{\d-1} f(y)dy\right).
\end{eqnarray} admits at most one fixed point $f\colon\R\to[0,1]$. Proposition \ref{pr:tight} will then guarantee the convergence $f_\lambda\xrightarrow[\lambda\to\infty]{} f$, which will in turn imply
\begin{eqnarray*}
\int_{-\lambda}^\lambda f_\lambda(x)\ln{f_\lambda(x)}dx & \xrightarrow[\lambda\to+\infty]{} & \int_{\R}f(x)\ln{f(x)}dx,
\end{eqnarray*}
by dominated convergence, thanks to the last inequalities in Proposition \ref{pr:uniform}. 

A quick inspection of the proof of Proposition \ref{pr:uniform} reveals that it remains valid when $\lambda=\infty$. In particular, any solution $f$ to (\ref{Tinfty}) must satisfy 
\begin{eqnarray}
\label{unif}
\max(f(x),1-f(-x)) & \leq & \exp\left(-\frac{x^\d}{e}\right),
\end{eqnarray}
for all $x\geq 0$. It also clear from (\ref{Tinfty}) that $f$ must be $(0,1)-$valued and  continuous. We will use those properties in the proofs below. 
\begin{lemma}
\label{lm:shift}
If $f,g$ solve (\ref{Tinfty}), then there exists $t\geq 0$ such that for all $x\in\R$, 
$$f(x+t)\leq g(x)\leq f(x-t).$$ 
\end{lemma}
\begin{proof} 
(\ref{unif}) ensures that for any $t\in\R$, $y\mapsto (1+|y|)(f(y-t)-g(y))$ is integrable on $\R$, so that by dominated convergence, 
\begin{eqnarray}
\label{lim}
\frac{1}{x^{\d-1}}\int_{-x}^{+\infty}(y+x)^{\d-1}\left(f(y-t)-g(y)\right) dy & \xrightarrow[x\to+\infty]{} &\Delta(t),
\end{eqnarray}
where 
\begin{eqnarray}
\label{Delta}
\Delta(t) & := & \int_{\R} \left(f(y-t)-g(y)\right) dy.
\end{eqnarray}
Observe that $t\mapsto\Delta(t)$ increases continuously from $-\infty$ to $+\infty$, as can be seen from the decomposition
\begin{eqnarray*}
\Delta(t) & = &  \int_{0}^{+\infty} (1- g(-y)-g(y))dy +\int_{-t}^{+\infty} f(y)dy -  \int_{t}^{+\infty} (1- f(-y))dy.
\end{eqnarray*}
In particular, we can find $t_0\geq 0$ such that $\Delta(-t_0)<0<\Delta(t_0)$. In view of (\ref{lim}), we deduce the existence of $a\geq 0$ such that for all $x\geq a$,
\begin{eqnarray}
\label{aa}
\int_{-x}^{+\infty}(y+x)^{\d-1}g(y) dy & \geq & 
 \int_{-x}^{+\infty}(y+x)^{\d-1}f(y+t_0)dy\\
 \label{bb}
 \int_{-x}^{+\infty}(y+x)^{\d-1}g(y) dy & \leq & 
 \int_{-x}^{+\infty}(y+x)^{\d-1}f(y-t_0)dy.
\end{eqnarray}
Applying $u\mapsto\exp(-\d u)$, we conclude that for all $x\geq a$,  
\begin{eqnarray}
\label{a}
 f(x+t_0) \ \leq & g(x) & \leq f(x-t_0).
\end{eqnarray}
In turn, this implies that  (\ref{aa})-(\ref{bb}) also hold when $x\leq -a$, so that (\ref{a}) actually holds for all $x$ outside $(-a,a)$. On the other hand, since $g$ is $(0,1)-$valued and $f$ has limits $0,1$ at $\pm\infty$, we can choose $t_1\geq 0$ large enough so that 
\begin{eqnarray*}
 f(-a+t_1) \ \leq \ g(a) & \leq & g(-a)\ \leq \ f(a-t_1).
\end{eqnarray*}
Since $f,g$ are non-increasing, this inequality implies that for all $x\in[-a,a]$,
\begin{eqnarray}
\label{b}
 f(x+t_1) \ \leq & g(x) & \leq \ f(x-t_1).
\end{eqnarray}
In view of (\ref{a})-(\ref{b}), taking $t:=\max(t_0, t_1)$ concludes the proof. 
\end{proof}

\begin{proof}[Proof of Proposition \ref{pr:tight}]
Let $f,g$ solve equation (\ref{Tinfty}) and let $t$ be the smallest non-negative number satisfying for all $x\in\R$,\begin{eqnarray}
\label{initial}
 f(x+t) & \leq \ g(x) \ \leq & f(x-t).
\end{eqnarray}
Note that $t$ exists by Lemma \ref{lm:shift} and the continuity of $f$. Now assume for a contradiction that $t>0$. Clearly, each of the two inequalities in (\ref{initial}) must be strict at some point $x\in\R$ (and hence on some open interval by continuity), otherwise we would have $g\geq f$ or $g\leq f$ and (\ref{Tinfty}) would then force $g=f$, contradicting the assumption that $t>0$. Consequently, the function $\Delta$ defined in (\ref{Delta}) must satisfy $\Delta(-t) <  0  <  \Delta(t).$
By continuity of $\Delta$, there must exists $t_0<t$ such that 
$\Delta(-t_0)<0<\Delta(t_0)$. As  we have already seen, this inequality implies  
\begin{eqnarray}
\label{t0}
 f(x+t_0) & \leq \ g(x) \ \leq & f(x-t_0),
\end{eqnarray}
for all $x$ outside some compact $[-a,a]$. In particular, we now see that the inequalities in (\ref{initial}) must be strict for all large enough $x$. Thus, for all $x\in\R$, 
\begin{eqnarray*}
\int_{-x}^{+\infty}(y+x)^{\d-1}g(y) dy & > & 
 \int_{-x}^{+\infty}(y+x)^{\d-1}f(y+t)dy\\
 \int_{-x}^{+\infty}(y+x)^{\d-1}g(y) dy & < & 
 \int_{-x}^{+\infty}(y+x)^{\d-1}f(y-t)dy.
\end{eqnarray*}
Applying $u\mapsto\exp(-\d u)$ now shows that the inequalities in (\ref{initial}) must actually be strict everywhere on $\R$, hence in particular on the compact $[-a,a]$. By uniform continuity, there must exists $t_1<t$ such that 
 \begin{eqnarray}
 \label{t1}
 f(x+t_1) & \leq \ g(x) \ \leq & f(x-t_1),
\end{eqnarray}
for all $x\in [-a,a]$. In view of (\ref{t0})-(\ref{t1}), the number $t':=\max(t_0,t_1)$  
now contradicts the minimality of $t$. 
\end{proof}

\bibliographystyle{plain}
\bibliography{draft}

\end{document}